\documentclass[12pt,a4paper]{amsart}
\usepackage{amsmath}
\usepackage{amsthm}
\usepackage{amssymb}
\usepackage{mathrsfs}
\usepackage[english]{babel}
\usepackage{verbatim}
\usepackage[shortlabels]{enumitem}

\newtheoremstyle{mio}%
	{}{} 
	{\itshape}{} 
	{\bfseries}{.}{ } 
	{#1 #2\thmnote{\mdseries~(\scshape #3)}} 
\theoremstyle{mio}
\newtheorem{teor}{Theorem}[section]
\newtheorem{cor}[teor]{Corollary}
\newtheorem{prop}[teor]{Proposition}
\newtheorem{lemma}[teor]{Lemma}
\newtheorem{defin}[teor]{Definition}

\newtheoremstyle{definition2}%
	{}{} 
	{}{} 
	{\bfseries}{.}{ } 
	{#1 #2\thmnote{\mdseries~ #3}} 
\theoremstyle{definition2}
\newtheorem{ex}[teor]{Example}
\newtheorem{oss}[teor]{Remark}

\newcommand{\insfracid}{\mathcal{F}}

\newcommand{\inssubmodqr}{\mathbf{F}}
\newcommand{\inssubmod}{\mathrm{SMod}}
\newcommand{\inssubflat}{\mathrm{SModFlat}}

\newcommand{\inssemistar}{\mathrm{SStar}}
\newcommand{\inssemistartf}{\inssemistar_f}

\newcommand{\inssemispectral}{\inssemistar_{sp}}
\newcommand{\inssemispectraltf}{\inssemistar_{f,sp}}

\DeclareMathOperator{\Spec}{Spec}

\DeclareMathOperator{\rad}{rad}
\newcommand{\ins}[1]{\mathbb{#1}}
\newcommand{\insC}{\ins{C}}
\newcommand{\Zar}{\mathrm{Zar}}
\newcommand{\Over}{\mathrm{Over}}
\newcommand{\inslocali}{\mathrm{LocOver}}
\newcommand{\Localiz}{\mathrm{Loc}}
\newcommand{\Overqr}{\mathrm{Over_{qr}}}
\newcommand{\Overflat}{\mathrm{Over_{flat}}}
\newcommand{\Oversloc}{\mathrm{Over_{sloc}}}

\newcommand{\spectral}[1]{\widetilde{#1}}

\newcommand{\qspec}[1]{\mathrm{QSpec}^{#1}}

\newcommand{\B}{\mathcal{B}}
\newcommand{\D}{\mathcal{D}}
\newcommand{\V}{\mathcal{V}}
\newcommand{\xcal}{{\boldsymbol{\mathcal{X}}}}
\newcommand{\scal}{{\boldsymbol{\mathcal{S}}}}
\newcommand{\ucal}{{\boldsymbol{\mathcal{U}}}}

\newcommand{\inverse}{{\mathrm{inv}}}
\newcommand{\cons}{\mathrm{cons}}

\title[Topological properties of localizations]{Topological properties of localizations, flat overrings and sublocalizations}
\author{Dario Spirito}
\address{Dipartimento di Matematica e Fisica, Universit\`a degli Studi
``Roma Tre'', Roma, Italy}
\email{spirito@mat.uniroma3.it}
\date{\today}
\thanks{This work was partially supported by {\sl GNSAGA} of {\sl Istituto Nazionale di Alta Matematica}.}

\begin{document}
\begin{abstract}
We study the set of localizations of an integral domain from a topological point of view, showing that it is always a spectral space and characterizing when it is a proconstructible subspace of the space of all overrings. We then study the same problems in the case of quotient rings, flat overrings and sublocalizations.
\end{abstract}

\maketitle

\section{Introduction}
The Zariski topology on the set $\Over(D)$ of overrings of an integral domain was introduced as a natural generalization of the Zariski topology on the space $\Zar(D)$ of valuation overrings of $D$ (called the \emph{Zariski space} of $D$), which in turn was introduced by Zariski in order to tackle the problem of resolution of singularities \cite{zariski_sing,zariski_comp}.

It has been proved that $\Over(D)$, like $\Zar(D)$, is a \emph{spectral space}, meaning that it is homeomorphic to the prime spectrum of a ring \cite[Proposition 3.5]{finocchiaro-ultrafiltri}. There are other subspaces of $\Over(D)$ that are always spectral: for example, this happens for the space of integrally closed overrings \cite[Proposition 3.6]{finocchiaro-ultrafiltri} and the space of local overrings \cite[Corollary 2.14]{surveygraz}.

In the last two cases, the role of $D$ in the definition of the space is merely to provide a setting ($\Over(D)$): that is, for an overring, being integrally closed or local (or a valuation domain, for the case of $\Zar(D)$) is a property completely independent from $D$. Indeed, with very similar proofs it is possible to generalize these results to the case of the spaces of rings comprised between two fixed rings (see e.g. \cite[Propositions 3.5 and 3.6]{finocchiaro-ultrafiltri} and \cite[Example 2.13]{surveygraz}), as well as using these methods to study spaces of modules \cite[Example 2.2]{olberding_topasp}. 

In this paper, we study four subspaces of $\Over(D)$ that are much more closely related to $D$; more precisely, such that, given an overring $T$, the belonging of $T$ to the space depends not on the properties of $T$ but rather on the relation between $D$ and $T$. In Section \ref{sect:loc} we shall start from the space of  localizations (at prime ideals); then we will consider the space of quotient rings (Section \ref{sect:qr}), sublocalizations of $D$ (i.e., intersection of localizations of $D$; Section \ref{sect:sloc}) and flat overrings (Section \ref{sect:flat}).

In each case, we will study two questions: under which conditions they are spectral spaces and under which condition they are closed in the constructible topology of $\Over(D)$. We shall answer completely these questions in the case of localizations (Theorem \ref{teor:Loc}) and quotient rings (Corollary \ref{cor:qr-spec} and Theorem \ref{teor:Overqr}); for sublocalizations we will find a sufficient condition (Theorem \ref{teor:radcolon-sloc}), while for flat overrings we will prove a characterization that is, however, very difficult to use (Proposition \ref{prop:flat-cons}). We shall also study the space of flat submodules of an $R$-module (for rings $R$ that are not necessarily integral domains) and the possibility of representing the space of sublocalizations of $D$ in a more topological way.

\section{Preliminaries}
\subsection{Spectral spaces}
A \emph{spectral space} is a topological space homeomorphic to the prime spectrum of a (commutative, unitary) ring (endowed with the Zariski topology). Spectral spaces can be characterized topologically as those spaces that are $T_0$ (i.e., such that for every pair of points at least one of them is contained in an open set not containing the other), compact, with a basis of open and compact subsets closed by finite intersections, and such that every nonempty irreducible closed subset has a generic point (i.e., it is the closure of a single point) \cite[Proposition 4]{hochster_spectral}.

If $X$ is a spectral space, the \emph{constructible topology} (or \emph{patch topology}) on $X$ (which we denote by $X^\cons$) is the coarsest topology such that the open and compact subspaces of the original topology are both open and closed. The space $X^\cons$ is always a spectral space, that is moreover Hausdorff and totally disconnected \cite[Theorem 1]{hochster_spectral}.

A subset $Y\subseteq X$ is said to be \emph{proconstructible} if it is closed, with respect to the constructible topology; in this case, the constructible topology on $Y$ coincides with the topology induced by the constructible topology on $X$, and $Y$ (with the original topology) is a spectral space (this follows from \cite[1.9.5(vi-vii)]{EGA4-1}). The converse does not hold, i.e., a subspace $Y$ of a spectral space $X$ may be spectral but not proconstructible; however, the following result holds.
\begin{lemma}\label{lemma:YX-cons}
Let $Y\subseteq X$ be spectral spaces. Suppose that there is a subbasis $\mathcal{B}$ of $X$ such that, for every $B\in\mathcal{B}$, both $B$ and $B\cap Y$ are compact. Then, $Y$ is a proconstructible subset of $X$.
\end{lemma}
\begin{proof}
The hypothesis on $\mathcal{B}$ implies that the inclusion map $Y\hookrightarrow X$ is a spectral map; by \cite[1.9.5(vii)]{EGA4-1}, it follows that $Y$ is a proconstructible subset of $X$.
\end{proof}

For further results about the constructible topology and the relation between ultrafilters and the constructible topology, see \cite{fontana_patch,fifolo_transactions,finocchiaro-ultrafiltri,surveygraz}.

\subsection{The space $\xcal(X)$}
Let $X$ be a spectral space. The \emph{inverse topology} on $X$ is the space $X^\inverse$ having, as a basis of closed sets, the open and compact subspaces of $X$; equivalently, it is the topology having as closed sets the subsets of $X$ that are compact and closed by generizations. The space $X^\inverse$ is again a spectral space. Following \cite{Xx}, we denote by $\xcal(X)$ the space of nonempty subsets of $X$ that are closed in the inverse topology; this space can be endowed with a topology having, as a basis of open sets, the sets of the form
\begin{equation*}
\ucal(\Omega):=\{Y\in\xcal(X)\mid Y\subseteq\Omega\},
\end{equation*}
as $\Omega$ ranges among the open and compact subspaces of $X$. Under this topology, $\xcal(X)$ is again a spectral space \cite[Theorem 3.2(1)]{Xx}.

If $X=\Spec(R)$ for some ring $R$, we set $\xcal(R):=\xcal(\Spec(R))$.

\subsection{Semistar operations}
Let $D$ be an integral domain with quotient field $K$, and let $\inssubmodqr(D)$ be the set of $D$-submodules of $K$. A \emph{semistar operation} on $D$ is a map $\star:\inssubmodqr(D)\longrightarrow\inssubmodqr(D)$ such that, for every $I,J\in\inssubmodqr(D)$ and every $x\in K$,
\begin{enumerate}
\item $I\subseteq I^\star$;
\item $(I^\star)^\star=I^\star$;
\item if $I\subseteq J$ then $I^\star\subseteq J^\star$;
\item $x\cdot I^\star=(xI)^\star$.
\end{enumerate}

A semistar operation is called \emph{spectral} if it is in the form $s_\Delta$ for some $\Delta\subseteq\Spec(D)$, where
\begin{equation*}
I^{s_\Delta}:=\bigcap\{ID_P\mid P\in\Delta\}
\end{equation*}
for every $I\in\inssubmodqr(D)$. If $\star$ is spectral, then $(I\cap J)^\star=I^\star\cap J^\star$ for every $I,J\in\inssubmodqr(D)$.

Starting from any semistar operations $\star$, we can define two maps $\star_f$ and $\spectral{\star}$ by putting, for every $I\in\inssubmodqr(D)$,
\begin{equation*}
I^{\star_f}=\bigcup\{J^\star\mid J\subseteq I,J\text{~is finitely generated}\} \end{equation*}
and
\begin{equation*}
I^{\spectral{\star}}:=\bigcup\{(I:E)\mid 1\in E^\star,E\text{~is finitely generated}\}.
\end{equation*}
Both $\star_f$ and $\spectral{\star}$ are semistar operations, and we always have $(\star_f)_f=\star_f$ and $\spectral{\spectral{\star}}=\spectral{\star}$. If $\star=\star_f$ then $\star$ is said to be \emph{of finite type}; on the other hand, $\star=\spectral{\star}$ if and only if $\star$ is spectral and of finite type.

If $\star=s_\Delta$ is a spectral operation, then $\star$ is of finite type if and only if $\Delta$ is compact \cite[Corollary 4.4]{topological-cons}.

The space $\inssemistar(D)$ of semistar operations on $D$ can be endowed with a topology having, as a basis of open sets, the sets of the form
\begin{equation*}
V_I:=\{\star\in\inssemistar(D)\mid 1\in I^\star\},
\end{equation*}
as $I$ ranges in $\inssubmodqr(D)$. In the induced topology, both the space $\inssemistartf(D)$ of finite-type operations and the space $\inssemispectraltf(D)$ of finite-type spectral operations are spectral (see \cite[Theorem 2.13]{topological-cons} for the former and \cite[Theorem 4.6]{spettrali-eab} for the latter). Moreover, $\inssemispectraltf(D)$ is homeomorphic to $\xcal(D)$ \cite[Proposition 5.2]{Xx}.

\subsection{The $t$-operation}
Let $D$ be an integral domain with quotient field $K$, and let $\star$ be a semistar operation on $D$. If $D^\star=D$, then the restriction of $\star$ to the set $\insfracid(D)$ of fractional ideals of $D$ is said to be a \emph{star operation} on $D$. A classical example of a star operation is the \emph{divisorial closure} (or $v$-operation), which is defined by $I^v:=(D:(D:I))$, where $(I:J):=\{x\in K\mid xJ\subseteq I\}$; the divisorial closure is the biggest star operation on $D$, in the sense that $I^\star\subseteq I^v$ for every star operation $\star$ and every $I\in\insfracid(D)$.

The \emph{$t$-operation} is the finite-type operation associated to the $v$-operation; that is, $t:=v_f$. The $t$-operation is the biggest finite-type star operation. The \emph{$w$-operation}, defined by $w:=\spectral{t}=\spectral{v}$, is the biggest spectral star operation of finite type.

If $\star$ is a star operation on $D$, a prime ideal $P$ of $D$ such that $P=P^\star$ is said to be a \emph{$\star$-prime}; the set of all $\star$-primes is called the \emph{$\star$-spectrum} and is denoted by $\qspec{\star}(D)$. If $\star=s_\Delta$ is a spectral star operation, then $\qspec{\star}(D)=\Delta^\downarrow=\{Q\in\Spec(D)\mid Q\subseteq P\text{~for some~}P\in\Delta\}$.

We always have $D=\bigcap\{D_P\mid P\in\qspec{t}(D)\}$.

See \cite[Chapter 32]{gilmer} for more properties of star operations.

\subsection{Overrings}
Let $D$ be an integral domain with quotient field $K$. An \emph{overring} of $D$ is a ring comprised between $D$ and $K$. The space $\Over(D)$ of the overrings of $D$ can be endowed with a topology having, as a basis of open sets, the sets of the form
\begin{equation*}
\B(x_1,\ldots,x_n):=\{T\in\Over(D)\mid x_1,\ldots,x_n\in T\}=\Over(D[x_1,\ldots,x_n]),
\end{equation*}
as $x_1,\ldots,x_n$ range in $K$. Under this topology, $\Over(D)$ is a spectral space \cite[Proposition 3.5]{finocchiaro-ultrafiltri}.

\section{Localizations}\label{sect:loc}
The first space we analyze is the space of localizations of an integral domain $D$ at its primes ideals, which we denote by $\Localiz(D)$; that is,
\begin{equation*}
\Localiz(D):=\{D_P\mid P\in\Spec(D)\}.
\end{equation*}

\begin{defin}
Let $D$ be an integral domain. We say that $D$ is \emph{rad-colon coherent} if, for every $x\in K\setminus D$, there is a finitely generated ideal $I$ such that $\rad(I)=\rad((D:_D x))$, i.e., if and only if $\D((D:_D x))$ is compact in $\Spec(D)$ for every $x\in K$.
\end{defin}

Obvious examples of rad-colon coherent domains are Noetherian domains or, more generally, domains with Noetherian spectrum. Another large class of such domains is the class of \emph{coherent domains}, i.e., domains where the intersection of two finitely generated ideals is still finitely generated; this follows from the fact that $(D:_Dx)=D\cap x^{-1}D$. In particular, this class contains all Pr\"ufer domains \cite[Proposition 25.4(1)]{gilmer}, or more generally the polynomial rings in finitely many variables over Pr\"ufer domains \cite[Corollary 7.3.4]{glaz-coherent}. See the following Example \ref{ex:Loc-nonconstr} for a domain that is not rad-colon coherent.

\begin{teor}\label{teor:Loc}
Let $D$ be an integral domain.
\begin{enumerate}[(a)]
\item\label{teor:Loc:sp} $\Localiz(D)$ is a spectral space.
\item\label{teor:Loc:proc} $\Localiz(D)$ is proconstructible in $\Over(D)$ if and only if $D$ is rad-colon coherent.
\end{enumerate}
\end{teor}
\begin{proof}
\ref{teor:Loc:sp} By \cite[Lemma 2.4]{dobbs_fedder_fontana}, the map
\begin{equation*}
\begin{aligned}
\lambda\colon\Spec(D) & \longrightarrow \Over(D)\\
P & \longmapsto D_P.
\end{aligned}
\end{equation*}
is a topological embedding whose image is exactly $\Localiz(D)$. In particular, since $\Spec(D)$ is a spectral space, so is $\Localiz(D)$.

\ref{teor:Loc:proc} We first note that
\begin{equation*}
\begin{array}{rcl}
\B(x)\cap\Localiz(D) & = & \{D_P\in\Localiz(D)\mid x\in D_P\}\\
& = & \{D_P\in\Localiz(D)\mid 1\in(D_P:x)\cap D\}=\\
& = & \{D_P\in\Localiz(D)\mid 1\in(D:_D:x)D_P\}=\\
& = & \{D_P\in\Localiz(D)\mid (D:_D:x)\subsetneq P\}=\lambda(\D((D:_Dx))).
\end{array}
\end{equation*}

Suppose $\Localiz(D)$ is proconstructible in $\Over(D)$. Since, for any $x\in K$, $\B(x)$ is also a proconstructible subspace of $\Over(D)$, then $\B(x)\cap\Localiz(D)$ is closed in $\Over(D)^\cons$; since the Zariski topology is weaker than the constructible topology, $\B(x)\cap\Localiz(D)$ must be compact in the Zariski topology. By the previous calculation, $\B(x)\cap\Localiz(D)=\lambda(\D(D:_Dx))$, and thus $\D((D:_D x))$ must be compact. Hence, $D$ is rad-colon coherent.

Conversely, suppose $D$ is rad-colon coherent. Then, each $\B(x)\cap\Localiz(D)$ is compact, and thus $\{\B(x)\cap\Localiz(D)\mid x\in K\}$ is a subbasis of compact subsets for $\Localiz(D)$; applying Lemma \ref{lemma:YX-cons} we see that $\Localiz(D)$ is a proconstructible subset of $\Over(D)$.
\end{proof}

As a first use of this theorem, we give an example of a domain that is not rad-colon coherent.
\begin{ex}\label{ex:Loc-nonconstr}
Let $D$ be an essential domain that is not a P$v$MD; that is, suppose that $D$ is the intersection of a family of valuation rings, each of which is a localization of $D$, but suppose that there is a $t$-prime ideal $P$ such that $D_P$ is not a valuation ring. Such a ring does indeed exists -- see \cite{ohm-essential}.

Let $\mathcal{E}$ be the set of prime ideals $P$ of $D$ such that $D_P$ is a valuation domain. Since $D$ is not a P$v$MD, not all $t$-primes are in $\mathcal{E}$. Since $\mathcal{E}\subseteq \qspec{t}(D)$ \cite[Lemma 3.17]{kang_pvmd}, we thus have $\mathcal{E}\subsetneq \qspec{t}(D)$. If $\mathcal{E}$ is compact, then $s_\mathcal{E}$ is a semistar operation of finite type on $D$; however, since $D$ is essential (and thus, by definition, $\bigcap\{D_P\mid P\in\mathcal{E}\}=D$) we have $D^{s_\mathcal{E}}=D$, and thus the restriction of $s_\mathcal{E}$ to the fractional ideals of $D$ is a spectral star operation of finite type, which implies that $I^{s_{\mathcal{E}}}\subseteq I^w$ for every finite-type operation. In particular,
\begin{equation*}
\mathcal{E}=\qspec{s_\mathcal{E}}(D)\supseteq \qspec{w}(D)\supseteq \qspec{t}(D),
\end{equation*}
and thus $\mathcal{E}=\qspec{t}(D)$, a contradiction. Therefore, $\mathcal{E}$ is not compact.

However, $\lambda(\mathcal{E})=\Localiz(D)\cap\Zar(D)$; if $\Localiz(D)$ were to be proconstructible in $\Over(D)$, so would be $\lambda(\mathcal{E})$ (since $\Zar(D)$ is always proconstructible). But this would imply that $\lambda(\mathcal{E})$ is, in particular, compact, a contradiction. Hence $\Localiz(D)$ is not proconstructible in $\Over(D)$, and $D$ is not rad-colon coherent.
\end{ex}

There are at least three natural ways to extend $\Localiz(D)$ to non-local overrings of $D$.

The first is by considering general localizations of $D$ (which we will call, for clarity, \emph{quotient rings}), that is, overrings in the form $S^{-1}D$ for some multiplicatively closed subsets $S$ of $D$. We denote this set by $\Overqr(D)$.

The second is through the set of \emph{flat} overrings of $D$ (that is, overrings that are flat when considered as $D$-modules). We denote this set by $\Overflat(D)$.

The third is by considering \emph{sublocalizations} of $D$, i.e., overrings that are intersection of localizations (or, equivalently, quotient rings) of $D$.  We denote this set by $\Oversloc(D)$.

It is well-known that $\Overqr(D)\subseteq\Overflat(D)\subseteq\Oversloc(D)$, and that both inclusions may be strict. For example, any overring of a Pr\"ufer domain is flat, but it need not be a quotient ring: in the case of Dedekind domains, this happens if and only if the class group of $D$ is torsion \cite[Corollary 2.6]{gilmer_qr} (more generally, a Pr\"ufer domain $D$ such that $\Overqr(D)=\Overflat(D)$ is said to be a \emph{QR-domain} -- see \cite[Section 27]{gilmer} or \cite[Section 3.2]{fontana_libro}). As for sublocalizations that are not flat, we shall give an example later (Example \ref{ex:flatqr}); see also \cite{well-centered}.

In all three cases, a natural question is to ask if (or when) the spaces are spectral, and if (or when) they are proconstructible in $\Over(D)$; moreover, we could ask if there is some construction through which we can represent them. We shall treat the case of quotient rings in Section \ref{sect:qr}, the case of sublocalizations in Section \ref{sect:sloc} and the case of flat overrings in Section \ref{sect:flat}.

A first result is a relation between their proconstructibility and the proconstructibility of $\Localiz(D)$.
\begin{prop}\label{prop:intersez-Loccons}
Let $D$ be an integral domain. If $\Overqr(D)$ or $\Overflat(D)$ is proconstructible, then $D$ is rad-colon coherent.
\end{prop}
\begin{proof}
Let $X$ be either $\Overqr(D)$ or $\Overflat(D)$, and let $\inslocali(D)$ be the space of local overrings of $D$. Then, $X\cap\inslocali(D)=\Localiz(D)$; since $\inslocali(D)$ is always proconstructible \cite[Corollary 2.14]{surveygraz}, if $X$ is proconstructible so is $\Localiz(D)$. By Theorem \ref{teor:Loc}\ref{teor:Loc:proc}, it follows that $D$ is rad-colon coherent.
\end{proof}

Note that $\Oversloc(D)\cap\inslocali(D)$ may not be equal to $\Localiz(D)$ -- see Example \ref{ex:flatqr}.

\section{Quotient rings}\label{sect:qr}
As localizations at prime ideals of $D$ can be represented through $\Spec(D)$, we can represent quotient rings by multiplicatively closed subsets; more precisely, there is a one-to-one correspondence between $\Overqr(D)$ and the set of multiplicatively closed subsets that are saturated. For technical reasons, it is more convenient to work with the complements of multiplicatively closed subsets.

\begin{defin}\label{defin:scal}
Let $R$ be a ring (not necessarily a domain). A \emph{semigroup prime} on $R$ is a nonempty subset $\mathscr{Q}\subseteq R$ such that:
\begin{enumerate}
\item for each $r\in R$ and for each $\pi\in\mathscr{Q}$, $r\pi\in\mathscr{Q}$;
\item for all $\sigma,\tau\in R\setminus\mathscr{Q}$, $\sigma\tau\in R\setminus \mathscr{Q}$;
\item $\mathscr{Q}\neq R$.
\end{enumerate}
\end{defin}

By \cite[(2.3)]{olberding_noetherianspaces}, a nonempty $\mathscr{Q}\subseteq R$ is a semigroup prime of $R$ if and only if it is a union of prime ideals, if and only if $R\setminus\mathscr{Q}$ is a saturated multiplicatively closed subset.

Let $\scal(R)$ denote the set of semigroup primes of a ring $R$. As in \cite{olberding_noetherianspaces} and in \cite{primi-sgr}, we endow $\scal(R)$ with the topology (which we call the \emph{Zariski topology}) whose subbasic closed sets have the form
\begin{equation*}
\V_\scal(x_1,\ldots,x_n):=\{\mathscr{Q}\in\scal(R)\mid x_1,\ldots,x_n\in\mathscr{Q}\},
\end{equation*}
as $x_1,\ldots,x_n$ ranges in $R$; equivalently, we can consider the subbasis of open sets
\begin{equation*}
\D_\scal(x_1,\ldots,x_n):=\scal(R)\setminus\V_\scal(x_1,\ldots,x_n)=\{\mathscr{Q}\in\scal(R)\mid x_i\notin\mathscr{Q}\text{~for some~}i\}.
\end{equation*}

We collect the properties of this topology of our interest in the next proposition.
\begin{prop}\label{prop:scal}
\cite[Propositions 2.3 and 3.1]{primi-sgr} Let $R$ be a ring and endow $\scal(R)$ with the Zariski topology.
\begin{enumerate}[(a)]
\item\label{prop:scal:subbasis} The family $\{\D_\scal(x)\mid x\in R\}$ is a basis of compact and open subsets of $\scal(R)$, which is closed by intersections.
\item\label{prop:scal:incl} The set-theoretic inclusion $\Spec(R)\hookrightarrow\scal(R)$ is a topological embedding.
\item\label{prop:scal:spectral} $\scal(R)$ is a spectral space.
\item\label{prop:scal:lambdaqr} Suppose $D$ is an integral domain. The map
\begin{equation*}
\begin{aligned}
\lambda_{qr}\colon\scal(D) & \longrightarrow \Over(D)\\
\mathscr{Q} & \longmapsto (R\setminus\mathscr{Q})^{-1}D.
\end{aligned}
\end{equation*}
is a topological embedding whose image is $\Overqr(D)$.
\end{enumerate}
\end{prop}

In particular, by points \ref{prop:scal:spectral} and \ref{prop:scal:lambdaqr} of the previous proposition we get immediately the following result.
\begin{cor}\label{cor:qr-spec}
$\Overqr(D)$ is a spectral space for every integral domain $D$.
\end{cor}

On the other hand, proconstructibility holds less frequently for $\Overqr(D)$ than it does for $\Localiz(D)$.
\begin{teor}\label{teor:Overqr}
Let $D$ be an integral domain with quotient field $K$. Then, $\Overqr(D)$ is proconstructible in $\Over(D)$ if and only if, for every $x\in K$, the ideal $\rad((D:_Dx))$ is the radical of a \emph{principal} ideal.
\end{teor}
\begin{proof}
As in the proof of Theorem \ref{teor:Loc}, we see that an overring $T$ is in $\B(x)\cap\Overqr(D)$ if and only if $T=\lambda_{qr}(\mathscr{Q})$ for some semigroup prime $\mathscr{Q}$ not containing $(D:_Dx)$. Moreover, we note that a semigroup prime contains an ideal $I$ if and only if it contains the radical of $I$.

Therefore, if each $\rad((D:_Dx))$ is the radical of a principal ideal, then each $\B(x)\cap\Overqr(D)$ is equal to $\lambda_{qr}(\D_\scal(y))$ for some $y\in D$. However, by Proposition \ref{prop:scal}\ref{prop:scal:subbasis}, $\D_\scal(y)$ is compact, and thus so is $\B(x)\cap\Overqr(D)$; by Lemma \ref{lemma:YX-cons}, $\Overqr(D)$ is proconstructible in $\Over(D)$.

Conversely, suppose there is a $x\in K$ be such that $I:=\rad((D:_Dx))$ is not the radical of a principal ideal.

\medskip

\emph{Claim 1}: let $y\in D$. Then, $D[y^{-1}]\in\B(x)$ if and only if $y\in I$.

\smallskip

If $x\in D[y^{-1}]$, then
\begin{equation}\label{eq:claim2}
1\in\left(D\left[y^{-1}\right]:_{D\left[y^{-1}\right]}x\right)= (D:_Dx)D\left[y^{-1}\right],
\end{equation}
since $D[y^{-1}]$ is flat over $D$.

If now $P\in\V(I)$ (i.e., $I\subseteq P$), then in particular $(D:_Dx)\subseteq P$, and so $PD[y^{-1}]=D[y^{-1}]$; it follows that $y\in P$. Since this happens for every $P\in\V(I)$ and $I$ is a radical ideal, $y\in I$.

Suppose now that $y\in I$. Then, every prime ideal containing $I$ explodes in $D[y^{-1}]$, and thus $ID[y^{-1}]=D[y^{-1}]$. Therefore, the same happens to $(D:_Dx)$, and so $x\in D[y^{-1}]$ (with the same calculation of \eqref{eq:claim2}, just backwards).

\medskip

Let now $\mathcal{U}:=\{\B(z^{-1})\mid z\in I\}$.

\medskip

\emph{Claim 2}: $\mathcal{U}$ is an open cover of $\B(x)\cap\Overqr(D)$.

\smallskip

Let $T\in\B(x)\cap\Overqr(D)$: then, $1\in(T:_Tx)=(D:_Dx)T$, and thus there are $d_1,\ldots,d_n\in(D:_Dx)$, $t_1,\ldots,t_n\in T$ such that $1=d_1t_1+\cdots+d_nt_n$. For every $i$, there is a $w_i\in D$ such that $w_i^{-1}\in T$ and $w_it_i\in D$; let $w:=w_1\cdots w_n$. Then, $w$ is invertible in $T$, and thus $D[w^{-1}]\subseteq T$, that is, $T\in\B(w^{-1})$; moreover,
\begin{equation*}
w=d_1wt_1+\cdots+d_nwt_n\in d_1D+\cdots+d_nD\subseteq(D:_Dx)\subseteq I,
\end{equation*}
and so $\B(w^{-1})\in\mathcal{U}$. Therefore, $\mathcal{U}$ is a cover of $\B(x)\cap\Overqr(D)$.

\medskip

\emph{Claim 3}: there are no finite subsets of $\mathcal{U}$ that cover $\B(x)\cap\Overqr(D)$.

\smallskip

Consider a finite subset $\mathcal{U}_0:=\{\B(z_1^{-1}),\ldots,\B(z_n^{-1})\}$ of $\mathcal{U}$, for some $z_1,\ldots,z_n\in I$. In particular, $\rad(z_iD)\subseteq I$ for every $I$; moreover, $\rad(z_iD)\neq I$ since $I$ is not the radical of any principal ideal. It follows that for every $i$ there is a prime ideal $P_i$ containing $z_i$ but not $I$. By prime avoidance, there is an $y\in I\setminus(P_1\cup\cdots\cup P_n)$; in particular, $D[y^{-1}]\in\B(x)\cap\Overqr(D)$.

We claim that $D[y^{-1}]\notin\B(z_i^{-1})$ for every $i$: indeed, $z_i\in P_i$, and $P_iD[y^{-1}]\neq D[y^{-1}]$. Therefore, $z_i$ is not invertible in $D[y^{-1}]$, and $z_i^{-1}\notin D[y^{-1}]$. Hence, $D[y^{-1}]$ is an element of $\B(x)\cap\Overqr(D)$ not contained in any element of $\mathcal{U}_0$, which thus is not a cover.

\medskip

Therefore, $\B(x)\cap\Overqr(D)$ is not compact; it follows that $\Overqr(D)$ is not proconstructible, as claimed.
\end{proof}

We remark that the first implication of the previous theorem follows also from \cite[Theorem 2.5]{well-centered} and the following Theorem \ref{teor:radcolon-sloc}.

\begin{cor}\label{cor:Noeth-qr}
Let $D$ be a Noetherian domain, and let $X^1(D)$ be the set of height-1 prime ideals of $D$. The following are equivalent:
\begin{enumerate}[(i)]
\item $\Overqr(D)$ is proconstructible in $\Over(D)$;
\item $D=\bigcap\{D_P\mid P\in X^1(D)\}$ and every $P\in X^1(D)$ is the radical of a principal ideal.
\end{enumerate}
\end{cor}
\begin{proof}
(i $\Longrightarrow$ ii) Suppose that $\Overqr(D)$ is proconstructible.

Let $Q$ be a prime $t$-ideal, and consider $A:=\bigcap\{D_P\mid P\in\D(Q)\}$. We claim that $A\neq D$: indeed, if $A=D$, then the map $\star:I\mapsto\bigcap\{ID_P\mid P\in\D(Q)\}$ would be a star operation of finite type (since $\D(Q)$ is compact) such that $Q^\star=D\nsubseteq Q=Q^t$, i.e., it would not be smaller than the $t$-operation, an absurdity. Hence, there is an $x\in A\setminus D$, and $\rad((D:_Dx))=Q$. By Theorem \ref{teor:Overqr}, $Q=\rad(yD)$ for some $y\in D$.

If $Q$ has not height 1, then this contradicts the Principal Ideal Theorem; thus, $\qspec{t}(D)=X^1(D)$, and $D=\bigcap\{D_P\mid P\in X^1(D)\}$.

(ii $\Longrightarrow$ i) Conversely, suppose that the two conditions hold; the first one implies that $\qspec{t}(D)=X^1(D)$ (since $X^1(D)$ is a compact subspace of $\Spec(D)$). For every $x\in K\setminus D$, $(D:_Dx)$ is a proper $t$-ideal, and thus its minimal primes are $t$-ideals, i.e., have height 1. However, $(D:_Dx)$ has only finitely many minimal primes, say $P_1,\ldots,P_n$, and by hypothesis $P_i=\rad(y_iD)$ for some $y_i\in D$; hence, $\rad((D:_Dx))$ is the radical of the principal ideal $y_1\cdots y_nD$. By Theorem \ref{teor:Overqr}, $\Overqr(D)$ is proconstructible.
\end{proof}

\begin{cor}
Let $D$ be a Krull domain, and let $X^1(D)$ be the set of height-1 prime ideals of $D$. Then, the following are equivalent:
\begin{enumerate}[(i)]
\item $\Overqr(D)$ is proconstructible in $\Over(D)$;
\item each $P\in X^1(D)$ is the radical of a principal ideal;
\item the class group of $D$ is a torsion group.
\end{enumerate}
\end{cor}
\begin{proof}
The equivalence between (i) and (ii) follows as in the previous corollary, noting that $D=\bigcap\{D_P\mid P\in X^1(D)\}$ holds for every Krull domain; the equivalence of (ii) and (iii) follows from the proof of Theorem 1 of \cite{waj-zak}.
\end{proof}

\section{Sublocalizations}\label{sect:sloc}
Our first result about $\Oversloc(D)$ shows a striking difference between the space of sublocalizations and the spaces we considered in the previous sections.
\begin{prop}\label{prop:sloc-consspec}
Let $D$ be an integral domain. Then, $\Oversloc(D)$ is a spectral space if and only if it is proconstructible in $\Over(D)$.
\end{prop}
\begin{proof}
If $\Oversloc(D)$ is proconstructible, then it is spectral. On the other hand, for every $x_1,\ldots,x_n\in K$, the intersection $\B(x_1,\ldots,x_n)\cap\Oversloc(D)$ is compact, since it has a minimum, namely the intersection of the localizations of $D$ that contain $x_1,\ldots,x_n$. Since $\{\B(x_1,\ldots,x_n)\cap\Oversloc(D)\mid x_1,\ldots,x_n\in K\}$ is a subbasis of $\Oversloc(D)$, by Lemma \ref{lemma:YX-cons} if $\Oversloc(D)$ is spectral then it is also proconstructible in $\Over(D)$.
\end{proof}

We are now tasked to study the spectrality of $\Oversloc(D)$. To this end, we use spectral semistar operations; more precisely, we use the fact that there is a map
\begin{equation*}
\begin{aligned}
\pi\colon\inssemispectral(D) & \longrightarrow \Oversloc(D)\\
\star & \longmapsto D^\star
\end{aligned}
\end{equation*}
that is continuous \cite[Proposition 3.2(2)]{surveygraz} and surjective (by definition of $\Oversloc(D)$). We shall use the following topological lemma.
\begin{lemma}\label{lemma:surj-Xspectral}
Let $\phi:X\longrightarrow Y$ be a continuous surjective map between two topological spaces. Suppose that:
\begin{enumerate}[(a)]
\item $X$ is spectral;
\item $Y$ is $T_0$;
\item there is a subbasis $\mathcal{C}$ of $Y$ such that, for every $C\in\mathcal{C}$, $\phi^{-1}(C)$ is compact.
\end{enumerate}
Then, $Y$ is a spectral space and $\phi$ is a spectral map.
\end{lemma}
\begin{proof}
Let $\Omega:=O_1\cap\ldots\cap O_m$ be a finite intersection of elements of $\mathcal{C}$. Then, $\phi^{-1}(\Omega)=\bigcap_i\phi^{-1}(O_i)$ is compact, since $X$ is spectral and each $\phi^{-1}(O_i)$ is compact by hypothesis; moreover, since $\phi$ is surjective, also $\Omega_i=\phi(\phi^{-1}(\Omega))$ is compact. Therefore, the set $\mathcal{C}_0$ of finite intersections of elements of $\mathcal{C}$ is a basis of compact subsets. If now $\Omega'$ is any open and compact subset of $Y$, then $\Omega$ is a finite union of elements of $\mathcal{C}_0$, and thus $\phi^{-1}(\Omega')$ is also compact.

The claim now follows from \cite[Proposition 9]{fontana_krr-abRs}.
\end{proof}

\begin{prop}\label{prop:spettrali-Oversloc}
Let $D$ be an integral domain. If $\inssemispectral(D)$ is a spectral space, then so is $\Oversloc(D)$.
\end{prop}
\begin{proof}
Let $\mathscr{B}:=\{\B(x)\cap\Oversloc(D)\mid x\in K\}$ be the canonical subbasis of $\Oversloc(D)$. Then,
\begin{equation*}
\begin{array}{rcl}
\pi^{-1}(\B(x)\cap\Oversloc(D)) & = & \{\star\in\inssemispectral(D)\mid x\in D^\star\}=\\
 & = & \{\star\in\inssemispectral(D)\mid 1\in x^{-1}D^\star\}=\\
 & = & \{\star\in\inssemispectral(D)\mid 1\in (x^{-1}D)^\star\}=\\
 & = & \{\star\in\inssemispectral(D)\mid 1\in (x^{-1}D\cap D)^\star\}=\\
 & = & V_{x^{-1}D\cap D}\cap\inssemispectral(D)=V_{(D:_D x)}\cap\inssemispectral(D).
\end{array}
\end{equation*}
However, $V_{(D:_D x)}\cap\inssemispectral(D)$ is compact since it has a minimum (explicitly, $s_{\D((D:_D x))}$). Hence, the map $\pi:\inssemispectral(D)\longrightarrow\Oversloc(D)$ satisfies the hypothesis of Lemma \ref{lemma:surj-Xspectral}, and thus $\Oversloc(D)$ is a spectral space.
\end{proof}

However, $\inssemispectral(D)$ is not, in general, a spectral space. To avoid this problem, we restrict $\pi$ to the space $\inssemispectraltf(D)$ (which is always spectral \cite[Theorem 4.6]{spettrali-eab}), obtaining the map $\pi_s:\inssemispectraltf(D)\longrightarrow\Oversloc(D)$; analogously to the previous proof, we need to show that $\pi_s$ is surjective and that $\pi_s^{-1}(\B(x)\cap\Oversloc(D))$ is compact. We claim that $D$ being rad-colon coherent is a sufficient condition for this to happen; we need a lemma.

\begin{lemma}\label{lemma:starf-starw}
Let $D$ be an integral domain, and let $\star$ be a spectral semistar operation on $D$.
\begin{enumerate}[(a)]
\item\label{lemma:starf-starw:F} If $\D(F\cap D)$ is a compact subset of $\Spec(D)$ for every finitely generated fractional ideal $F$ of $D$, then $\star_f=\spectral{\star}$.
\item\label{lemma:starf-starw:rc} If $D$ is rad-colon coherent, then $D^{\star_f}=D^{\spectral{\star}}$.
\end{enumerate}
\end{lemma}

Note that the equality $\star_f=\spectral{\star}$ may actually fail; see \cite[p.2466]{anderson_two_2000}.

\begin{proof}
\ref{lemma:starf-starw:F} Since $\star_f$ and $\spectral{\star}$ are of finite type, it is enough to show that $F^{\star_f}=F^{\spectral{\star}}$ if $F$ is finitely generated. The containment $F^{\spectral{\star}}\subseteq F^{\star_f}$ always holds; suppose $x\in F^{\star_f}$. Then, since $F^{\star_f}\subseteq F^\star$, we have $x\in F^\star$. Consider $I:=x^{-1}F\cap D$. Then, $xI=F\cap xD\subseteq F$. Moreover, 
\begin{equation*}
I^\star=(x^{-1}F\cap D)^\star=x^{-1}F^\star\cap D^\star
\end{equation*}
since $\star$ is spectral, and thus $1\in I^\star$. Since $x^{-1}F$ is finitely generated, by hypothesis $\D(I)$ is compact, and thus there is a finitely generated ideal $J$ of $D$ such that $\rad(I)=\rad(J)$; passing, if needed, to a power of $J$, we can suppose $J\subseteq I$, so that $xJ\subseteq xI\subseteq F$. For any spectral operation $\sharp$, $\rad(A)=\rad(B)$ implies that $1\in A^\sharp$ if and only if $1\in B^\sharp$; therefore, $1\in J^\star$, and thus $x\in(F:J)\subseteq F^{\spectral{\star}}$, and $x\in F^{\spectral{\star}}$. Hence, $\star_f=\spectral{\star}$, as requested.

\ref{lemma:starf-starw:rc} It is enough to repeat the proof of the previous point by using $F=D$, and noting that $\D(x^{-1}D\cap D)$ is compact since $D$ is rad-colon coherent.
\end{proof}

\begin{teor}\label{teor:radcolon-sloc}
Let $D$ be an integral domain. If $D$ is rad-colon coherent, then $\Oversloc(D)$ is a spectral space.
\end{teor}
\begin{proof}
Suppose $D$ is rad-colon coherent. If $T\in\Oversloc(D)$, then there is a $\sharp\in\inssemispectral(D)$ such that $T=D^\sharp$; since $D$ is $D$-finitely generated, moreover, we have $D^\sharp=D^{\sharp_f}$. By Lemma \ref{lemma:starf-starw}\ref{lemma:starf-starw:rc}, $D^{\sharp_f}=D^{\spectral{\sharp}}$; but $\spectral{\sharp}\in\inssemispectraltf(D)$, and thus $\pi_s$ is surjective.

As in the proof of Proposition \ref{prop:spettrali-Oversloc},
\begin{equation*}
\pi_s^{-1}(\B(x)\cap\Oversloc(D))=V_{(D:_D x)}\cap\inssemispectraltf(D),
\end{equation*}
which is compact since it has a minimum ($s_{\D((D:_Dx))}$). Since $\inssemispectraltf(D)$ is a spectral space \cite[Theorem 4.6]{spettrali-eab}, by  Lemma \ref{lemma:surj-Xspectral} $\Oversloc(D)$ is spectral.
\end{proof}

\begin{cor}
If $D$ is a domain with Noetherian spectrum (in particular, if $D$ is Noetherian) then $\Oversloc(D)$ is a spectral space.
\end{cor}

Note that it is not hard to see that, if $\D(J)$ is not compact in $\Spec(D)$, then $V_J\cap\inssemispectraltf(D)$ is actually not compact; therefore, the proof of Theorem \ref{teor:radcolon-sloc} cannot easily be further generalized.

Another natural question is whether $\pi_s$ is injective; however, this is usually false. For example, if $\Delta$ is any subset of $\Spec(D)$ containing the $t$-spectrum, then $\pi_s(s_\Delta)=D$. Thus, $\pi_s$ does not give a way to ``represent'' $\Oversloc(D)$ like $\Spec(D)$ does for $\Localiz(D)$ and $\scal(D)$ for $\Overqr(D)$. To circumvent this problem, we shall use, instead of the whole spectrum, the $t$-spectrum; note that $\qspec{t}(D)$ is a proconstructible subspace of $\Spec(D)$ \cite[Proposition 2.5]{intD-PvMD}, so a spectral space, and thus the space $\xcal(\qspec{t}(D))$ is defined and spectral.

Consider the map
\begin{equation*}
\begin{aligned}
\pi_t\colon\xcal(\qspec{t}(D)) & \longrightarrow\Oversloc(D) \\
\Delta & \longmapsto D^{s_\Delta}.
\end{aligned}
\end{equation*}
Note that, if $D$ is rad-colon coherent, $\pi_t$ is continuous and spectral, since it is the composition of the spectral inclusion $\xcal(\qspec{t}(D))\hookrightarrow\xcal(D)$ (\cite[Proposition 4.1]{Xx}, noting the inclusion $\qspec{t}(S)\hookrightarrow\Spec(D)$ is spectral since $\qspec{t}(D)$ is proconstructible), the homeomorphism $\xcal(D)\longrightarrow\inssemispectraltf(D)$ and the map $\pi_s:\inssemispectraltf(D)\longrightarrow\Over(D)$ (which is spectral by Lemma \ref{lemma:surj-Xspectral} and the proof of Theorem \ref{teor:radcolon-sloc}).

We first show that, using $\pi_t$, we do not lose anything.
\begin{prop}\label{prop:imm-pit}
Let $D$ be an integral domain. Then:
\begin{enumerate}[(a)]
\item\label{prop:imm-pit:intersez} for any $\Delta,\Lambda\in\xcal(D)$, if $\Delta\cap\qspec{t}(D)=\Lambda\cap\qspec{t}(D)$ then $\pi_s(s_\Delta)=\pi_s(s_\Lambda)$;
\item\label{prop:imm-pit:surj} $\pi_s(\inssemispectraltf(D))=\pi_t(\xcal(\qspec{t}(D)))$.
\end{enumerate}
\end{prop}
\begin{proof}
It is enough to show that, for every $\Delta\in\xcal(D)$, $\pi_s(\Delta)=\pi_s(\Delta_0)$, where $\Delta_0:=\Delta\cap\qspec{t}(D)$. Let $T:=\pi_s(s_\Delta)$; then, since $\Delta$ is a proconstructible subset of $\Spec(D)$, also $\Delta_0$ is proconstructible. In particular, $\Delta_0$ is compact and closed by generizations relative to $\qspec{t}(D)$, and so it belongs to $\xcal(\qspec{t}(D))$. We claim that $T=\pi_t(\Delta_0)$.

Indeed, let $P\in\Delta$. Then, $t_P:ID_P\mapsto I^tD_P$ is a star operation of finite type on $D_P$ (see \cite{twostar}), and $QD_P$ is a maximal $t_P$-ideal if and only if $Q$ is maximal among the $t$-prime ideals contained in $P$. Hence, $D_P=\bigcap\{D_Q\mid Q\subseteq P,Q=Q^t\}$, and
\begin{equation*}
T=\bigcap\{D_Q\mid Q=Q^t,Q\subseteq P\text{~for some~}P\in\Delta\}.
\end{equation*}
The set of primes on the right hand side is exactly $\Delta_0$. Therefore, $T=\pi_t(\Delta_0)\in\pi_t(\xcal(\qspec{t}(D)))$, and \ref{prop:imm-pit:intersez} is proved.

Moreover, this also shows that $\pi_s(\inssemispectraltf(D))\subseteq\pi_t(\xcal(\qspec{t}(D)))$; since the other inclusion is obvious, \ref{prop:imm-pit:surj} holds.
\end{proof}

The $t$-spectrum is much less redundant than $\Spec(D)$: indeed, if $D=\bigcap\{D_P\mid P\in\Delta\}$ for some compact $\Delta\subseteq\qspec{t}(D)$, then $\Delta$ must contain the $t$-maximal ideals, since $t$ is the biggest star operation of finite type. In general, $\pi_t$ is not always injective; however, when this happens then $\pi_t$ is also a homeomorphism, as the next proposition shows.
\begin{prop}\label{prop:pit-Omef}
Let $D$ be a rad-colon coherent domain. Then, the following are equivalent:
\begin{enumerate}[(i)]
\item\label{prop:pit-Omef:Omef} $\pi_t$ is a homeomorphism;
\item\label{prop:pit-Omef:inj} $\pi_t$ is injective;
\item\label{prop:pit-Omef:intersez} if $\Delta,\Lambda\in\xcal(D)$ are such that $\pi_s(s_\Delta)=\pi_s(s_\Lambda)$, then $\Delta\cap\qspec{t}(D)=\Lambda\cap\qspec{t}(D)$.
\end{enumerate}
\end{prop}
\begin{proof}
The implication (i $\Longrightarrow$ ii) is obvious; the equivalence between \ref{prop:pit-Omef:inj} and \ref{prop:pit-Omef:intersez} follows from Proposition \ref{prop:imm-pit}. 

Suppose now that $\pi_t$ is injective; then, $\pi_t$ is bijective (since it is also surjective by Theorem \ref{teor:radcolon-sloc}, being $D$ rad-colon coherent), continuous and spectral. Clearly, if $\Delta\supseteq\Lambda$ then $\pi_t(\Delta)\subseteq\pi_t(\Lambda)$. Conversely, suppose $\pi_t(\Delta)\subseteq\pi_t(\Lambda)$: then, $T:=\bigcap\{D_P\mid P\in\Delta\}\subseteq\bigcap\{D_Q\mid Q\in\Lambda\}$, and thus $T\subseteq D_Q$ for every $Q\in\Lambda$. Hence, $\pi_t(\Delta)=\pi_t(\Delta\cup\Lambda)$, and by the injectivity of $\pi_t$ is must be $\Delta=\Delta\cup\Lambda$, i.e., $\Lambda\subseteq\Delta$. Therefore, $\pi_t$ is also an order isomorphism (in the order induced by the respective topologies of $\xcal(\qspec{t}(D))$ and $\Oversloc(D))$; by \cite[Proposition 15]{hochster_spectral}, $\pi_t$ is a homeomorphism.
\end{proof}

A prime ideal $P$ of $D$ is \emph{well-behaved} if $PD_P$ is $t$-closed in $D_P$ \cite{wellbehaved}; this is equivalent to $D_P$ being a DW-domain, i.e., to the fact that, on $D_P$, the $w$-operation coincides with the identity (this follows from \cite[Proposition 2.2]{mimouni-DW}). A domain is called \emph{well-behaved} if every $t$-prime ideal is well-behaved; examples of well-behaved domains are Noetherian domains, Krull domains and domains where every $t$-prime ideal has height 1.
\begin{prop}\label{prop:pit-inj-car}
Let $D$ be an integral domain. Then, $D$ is well-behaved if and only if the map $\pi_t:\xcal(\qspec{t}(D))\longrightarrow\Oversloc(D)$ is injective.
\end{prop}
\begin{proof}
Suppose $\pi_t$ is injective, and let $P\in\qspec{t}(D)$ and $\Delta:=\qspec{t}(D_P)$. Then, $\Delta$ is compact (being proconstructible in $\Spec(D_P)$), and thus $\Delta\cap D:=\{Q\cap D\mid P\in\Delta\}$ is a compact subspace of $\qspec{t}(D)$, since it is the continuous image of $\Delta$ under the canonical map $\Spec(D_P)\longrightarrow\Spec(D)$. If $PD_P\notin\Delta$, then $P\notin\Delta\cap D$; however,
\begin{equation*}
\pi_t(\Delta\cap D)=\bigcap\{D_{Q\cap D}\mid Q\in\Delta\}=\bigcap\{(D_P)_Q\mid Q\in\Delta\}=D_P,
\end{equation*}
with the last equality coming from the properties of the $t$-spectrum. If we denote by $\Lambda_1$ the closure in the inverse topology of $\qspec{t}(D)$ of $\Delta\cap D$, and by $\Lambda_2$ the closure of $(\Delta\cap D)\cup\{P\}$, we have thus $\pi_t(\Lambda_1)=\pi_t(\Lambda_2)$ while $\Lambda_1\neq\Lambda_2$, against the injectivity of $\pi_t$.

On the other hand, suppose $D$ is well-behaved. Suppose $\pi_t(\Delta)=\pi_t(\Lambda)=:T$ for some $\Delta,\Lambda\in\xcal(\qspec{t}(D))$, $\Delta\neq\Lambda$, and let $P\in\Delta\setminus\Lambda$. By \cite[Lemma 2.4]{dobbs_fedder_fontana}, the subspace $\{D_Q\mid Q\in\Lambda\}\subseteq\Over(D)$ is compact; then,
\begin{equation*}
D_P=D_PT=D_P\bigcap_{Q\in\Lambda}D_Q=\bigcap_{Q\in\Lambda}D_PD_Q,
\end{equation*}
with the last equality coming from \cite[Corollary 5]{compact-intersections}. The family $\{D_PD_Q\mid Q\in\Lambda\}$ is again compact \cite[Lemma 4]{compact-intersections}; thus, $\star:I\mapsto\bigcap_{Q\in\Lambda}ID_PD_Q$ is a finite-type spectral semistar operation such that $D^\star=D_P$, and thus it restricts to a finite-type \emph{star} operation $\star'$ on $D_P$. Since $PD_P$ is $t$-closed, and $\star'$ is of finite type, $(PD_P)^{\star'}$ must be equal to $PD_P$; however,
\begin{equation*}
P^{\star'}=P^\star=\bigcap_{Q\in\Lambda}PD_QD_P=\bigcap_{Q\in\Lambda}D_QD_P=D_P,
\end{equation*}
since $P\nsubseteq Q$ for every $Q\in\Lambda$. This is a contradiction, and $\pi_t$ is injective.
\end{proof}

\begin{oss}\label{oss:wellb}
~\begin{enumerate}
\item There are examples of integral domains that are not well-behaved (see \cite[Section 2]{wellbehaved} or \cite[Example 1.4]{fintcar}), and thus $\pi_t$ is not always injective.
\item\label{oss:wellb:spec} It would be tempting to substitute the space $\xcal(\qspec{t}(D))$ with $\xcal(\Delta)$, where $\Delta$ is the set of well-behaved $t$-prime ideals of $D$. However, $\Delta$ may not be compact and thus, \emph{a fortiori}, may not be a spectral space. For example, consider a domain $D$ and a prime ideal $Q$ that is a maximal $t$-ideal (that is, $P$ is maximal among the ideals $I$ such that $I=I^t$) but not well-behaved. (An explicit example is $E+XE_S[X]$, where $E$ is the ring of entire functions, $X$ is an indeterminate and $S$ is the set of finite products of elements of the form $Z-\alpha$, as $\alpha$ ranges in $\insC$; see \cite[Example 2.6, Section 4.1 and Proposition 4.3]{zaf_gcd}.) Let $\Lambda$ be the set of prime ideals that are associated to some principal ideal; then, $P\in\Lambda$ if and only if $P$ is minimal over the ideal $(bD:_DaD)$, for some $a,b\in D$.

Since a principal ideal is $t$-closed, so is $(bD:_DaD)=\frac{b}{a}D\cap D$; moreover, a minimal prime over a $t$-ideal is again a $t$-ideal, and thus $\Lambda\subseteq\qspec{t}(D)$. Moreover, if $P\in\Lambda$ then $PD_P$ will be associated to a principal ideal of $D_P$ (if $P$ is minimal over $(bD:_DaD)$, then $PD_P$ is minimal over $(bD:_DaD)D_P=(bD_P:_{D_P}aD_P)$). Hence, each prime of $\Lambda$ is well-behaved, and $\Lambda\subseteq\Delta$.

By \cite{associatedprimes-princid}, we have $D=\bigcap\{D_P\mid P\in\Lambda\}$, and thus also $D=\bigcap\{D_P\mid P\in\Delta\}$. If $\Delta$ were compact, it would define a finite-type star operation $\star:I\mapsto\bigcap\{ID_P\mid P\in\Delta\}$ such that $Q^\star=D$. On the other hand, we should have $\star\leq t$ and thus $Q^\star\subseteq Q^t=Q$, a contradiction. Hence, $\Delta$ is not compact.
\end{enumerate}
\end{oss}

Recall that a domain is \emph{$v$-coherent} if, for any ideal $I$, $(D:I)=(D:J)$ for some finitely generated ideal $J$.

\begin{cor}\label{cor:pit-inj-vcoer}
Let $D$ be a $v$-coherent domain. Then, $\pi_t$ is injective.
\end{cor}
\begin{proof}
Since $D$ is $v$-coherent, $(ID_Q)^t=I^tD_Q$ for every ideal $I$ of $D$ \cite[proof of Proposition 4.6]{twostar} and every $Q\in\Spec(D)$; thus, if $P\in\qspec{t}(D)$ then $(PD_P)^t=P^tD_P=PD_P$. By Proposition \ref{prop:pit-inj-car}, $\pi_t$ is injective.
\end{proof}

\section{Flat overrings}\label{sect:flat}
The space $\Overflat(D)$ of flat overrings of $D$ is much more mysterious than $\Overqr(D)$ and $\Oversloc(D)$, and we are not able to characterize when it is spectral or proconstructible. The main theorem of this section is the following partial result.
\begin{prop}\label{prop:flat-cons}
Let $D$ be an integral domain. Then, $\Overflat(D)$ is a proconstructible subspace of $\Over(D)$ if and only if $\Overflat(D)\cap\B(x_1,\ldots,x_n)$ is compact for every $x_1,\ldots,x_n\in K$.
\end{prop}
\begin{proof}
If $\Overflat(D)$ is proconstructible, the compactness of $\Overflat(D)\cap\B(x_1,\ldots,x_n)$ follows like in the proof of Proposition \ref{prop:sloc-consspec}.

Suppose that the compactness property holds, and let $x_1,\ldots,x_n\in K$. Consider the canonical subbasis $\mathcal{S}:=\{\B(x)\cap X\mid x\in K\}$ of $X:=\Overflat(D)$. By \cite[Proposition 3.3]{finocchiaro-ultrafiltri} and \cite[Theorem 8]{fontana_patch} (or \cite[Corollary 2.17]{finocchiaro-ultrafiltri}), we need to show that, for every ultrafilter $\mathscr{U}$ on $X$, the ring $A_\mathscr{U}:=\{x\in K\mid\B(x)\cap X\in\mathscr{U}\}$ is flat.

Take $a_1,\ldots,a_n\in D$, $x_1,\ldots,x_n\in A_\mathscr{U}$ such that $a_1x_1+\cdots+a_nx_n=0$. For all $C\in\Overflat(D)\cap\B(x_1,\ldots,x_n)$, by the equational characterization of flatness (see e.g. \cite[Theorem 7.6]{matsumura} or \cite[Corollary 6.5]{eisenbud}) there are $b_{jk}^{(C)}\in D$, $y_k^{(C)}\in C$ such that
\begin{equation}\label{eq:flat-equational}
\begin{cases}
0=a_1b_{1k}^{(C)}+\cdots+a_nb_{nk}^{(C)} & \text{for all~}k\\
x_i=b_{i1}^{(C)}y_1^{(C)}+\cdots+b_{iN}^{(C)}y_N^{(C)} & \text{for all~}i.
\end{cases}
\end{equation}
Let $\Omega(C):=\B(y_1^{(C)},\ldots,y_{n_C}^{(C)})$. Then, the family of the $\Omega(C)$ is an open cover of $\Overflat(D)\cap\B(x_1,\ldots,x_n)$. Hence, there is a finite subcover $\{\Omega(C_1),\ldots,\Omega(C_n)\}$; by the properties of ultrafilters, it follows that $\Omega(C_j)\in\mathscr{U}$ for some $j$. Thus, $y_i^{(C_j)}\in A_\mathscr{U}$ for all $i$; then, \eqref{eq:flat-equational} holds in $A_\mathscr{U}$. Hence, applying again the equational criterion, $A_\mathscr{U}$ is flat.
\end{proof}

\begin{cor}\label{cor:sloc=flat}
Let $D$ be an integral domain such that $\Overflat(D)=\Oversloc(D)$. Then, $\Overflat(D)$ is a proconstructible subset of $\Over(D)$. In particular, $D$ is rad-colon coherent.
\end{cor}
\begin{proof}
It is enough to note that $\Oversloc(D)\cap\B(x_1,\ldots,x_n)$ has always a minimum, and apply Proposition \ref{prop:flat-cons}.
\end{proof}

\begin{ex}\label{ex:flatqr}
The space of flat overrings can be spectral even if it is not proconstructible.

Let $K$ be a field, and let $D:=K[[X^2,X^3,XY,Y]]$; that is, $D$ is the set of the power series in two variables over $K$ without the monomial corresponding to $X$. Then, $D$ is a two-dimensional local Noetherian domain; its integral closure is $A:=K[[X,Y]]=D[X]$, which is also equal to the intersection of the localizations at the height-1 primes of $D$. (In particular, $A$ is a local sublocalization of $D$ that is not a localization.) By Corollary \ref{cor:pit-inj-vcoer}, it is easy to see that the sublocalizations of $D$ are $D$ itself and the intersections $T(\Delta):=\bigcap\{D_P\mid P\in\Delta\}$, as $\Delta$ ranges among the subsets of $X^1(D):=\{P\in\Spec(D)\mid P\text{~has height~}1\}$.

A power series $\phi:=\sum_{i,j\geq 0}a_{ij}X^iY^j$ is invertible in $A$ if and only if $a_{00}\neq 0$; hence, if $\phi\in A$ is not invertible then $\phi^2\in D$. Since every height-1 prime ideal of $A$ is principal (being $A$ a unique factorization domain) and the canonical map $\Spec(A)\longrightarrow\Spec(D)$ is surjective, every height-1 prime ideal of $D$ is the radical of a principal ideal (if $P=Q\cap D$, for $Q\in\Spec(A)$, $Q=\phi A$, then $P$ is the radical of $\phi^2D$). Hence, $T(\Delta)$ is a quotient ring of $D$ for every $\Delta\subsetneq X^1(D)$; in particular, they are all flat. Hence, $\Overqr(D)=\Overflat(D)$ is spectral; however, $(D:_DX)$ is equal to the maximal ideal of $D$, which cannot be the radical of a principal ideal since it is of height 2. By Theorem \ref{teor:Overqr}, $\Overqr(D)$ (and so $\Overflat(D)$) is not proconstructible.
\end{ex}

The space $\Overflat(D)$ is, however, amenable to generalizations. Indeed, if $R$ is a ring and $M$ is an $R$-module, then the set $\inssubmod_R(M)$ of $R$-submodules of $M$ can be endowed with a topology (called the \emph{Zariski topology}) whose basic open sets are of the form
\begin{equation*}
\D(x_1,\ldots,x_n):=\{N\in\inssubmod_R(M)\mid x_1,\ldots,x_n\in N\},
\end{equation*}
as $x_1,\ldots,x_n$ vary in $M$. Under this topology, $\inssubmod_R(M)$ is a spectral space \cite[Example 2.2(2)]{olberding_topasp}; moreover, if $D$ is an integral domain with quotient field $K$, then the Zariski topology on $\Over(D)$ is exactly the restriction of the Zariski topology on $\inssubmod_D(K)=\inssubmodqr(D)$, and $\Over(D)$ is proconstructible in $\inssubmodqr(D)$.

We can consider on $\inssubmod_R(M)$ the subspace $\inssubflat_R(M)$ consisting of all flat $R$-submodules of $M$. Surprisingly, in many cases spectrality and proconstructibility of $\inssubflat_R(M)$ are equivalent.
\begin{prop}\label{prop:inssbuflat}
Let $R$ be a ring and $M$ be an $R$-module; suppose that $R$ is an integral domain or that $M$ is torsion-free. Then, $\inssubflat_R(M)$ is a spectral space if and only if it is proconstructible in $\inssubmod_R(M)$.
\end{prop}
\begin{proof}
Clearly if $\inssubflat_R(M)$ is proconstructible in $\inssubmod_R(M)$ then it is spectral.

Conversely, suppose that $Y:=\inssubflat_R(M)$ is spectral. By Lemma \ref{lemma:YX-cons}, $Y$ is proconstructible if and only if $\Omega\cap Y$ is compact for every $\Omega$ in some subbasis of $\inssubmod_R(M)$; since $\D(x_1,\ldots,x_n)=\D(x_1)\cap\cdots\cap\D(x_n)$ for every $x_1,\ldots,x_n\in M$, we can consider the subbasis $\{\D(x)\cap Y\mid x\in M\}$. By definition, $\D(x)\cap Y:=\{N\in Y\mid x\in Y\}$.

Let $x\in M$. If $x$ has no torsion (so, in particular, if $M$ is torsion-free), then the principal submodule $\langle x\rangle$ is isomorphic to $R$, which is flat; thus, $\D(x)\cap Y$ has a minimum, namely $\langle x\rangle$, and $\D(x)\cap Y$ is compact. On the other hand, if $R$ is an integral domain, then every flat $R$-module is torsion-free \cite[I.2, Proposition 3]{bourbaki_ac}; thus, if $x$ has torsion then no module containing $x$ can be flat, and so $\D(x)\cap Y$ must be empty (and in particular compact).

In all the cases considered, it follows that $\inssubflat_R(M)$ is proconstructible in $\inssubmod_R(M)$.
\end{proof}

\begin{cor}
Let $D$ be an integral domain with quotient field $K$, and suppose that $D$ is not rad-colon coherent. Then, $\inssubflat_D(K)$ is not a spectral space.
\end{cor}
\begin{proof}
The space $\Over(D)$ is proconstructible in $\inssubmod_D(K)$ \cite[Example 2.2(5)]{olberding_topasp}, and thus $\Overflat(D)$ is proconstructible in $\Over(D)$ if and only if it is proconstructible in $\inssubmod_D(K)$. If $\inssubflat_D(K)$ were spectral, by Proposition \ref{prop:inssbuflat}, it would follow that it is proconstructible in $\inssubmod_D(K)$; thus, also the intersection $\Over(D)\cap\inssubmod_D(K)=\Overflat(D)$ would be proconstructible in $\inssubmod_D(K)$.

However, if $D$ is not rad-colon coherent then $\Overflat(D)$ is not proconstructible in $\Over(D)$ (Proposition \ref{prop:intersez-Loccons}); hence, $\inssubflat_D(K)$ cannot be spectral.
\end{proof}

\end{document}